\documentclass[twoside,11pt]{amsart}

\usepackage[cmtip,all]{xy}
\usepackage{amssymb, xspace, enumerate, times, color, hyperref}

\usepackage[capitalize]{cleveref}
\usepackage{graphicx}
\input xypic
\input xy
\xyoption{all}
\def\sqr#1#2{{\vcenter{\hrule height.#2pt
        \hbox{\vrule width.#2pt height#1pt \kern#1pt
                \vrule width.#2pt}
        \hrule height.#2pt}}}

\newtheorem{Theorem}{Theorem}[section]

\newtheorem{Proposition}[Theorem]{Proposition}
\newtheorem{Conjecture}[Theorem]{Conjecture}

\newtheorem{Assumptions and Discussion}[Theorem]{Assumptions and Discussion}

\newtheorem{Construction}[Theorem]{Construction}
\newtheorem{Remark}[Theorem]{Remark}

\newtheorem{Definition}[Theorem]{Definition}

\def\m{{\mathfrak m}}

\def\q{{\mathfrak q}}
\def\p{{\mathfrak p}}

\newcommand{\ovl}[1]{\overline{#1}}

\newcommand{\wdt}[1]{\widetilde{#1}}

\newcommand{\symp}[1]{#1^{(\ell)}}

\newcommand{\supp}[1]{{\rm supp}(#1)}

\def\ZZ{{\mathbb Z}}
\def\NN{{\mathbb N}}

\def\kk{{\mathbb K}}
\def\Tor{{\textrm{Tor}}}

\newcommand{\B}{\mathcal{B}}

\newcommand{\F}{\mathcal{F}}

\def\Llra{\Longleftrightarrow}

\def\Lra{\Longrightarrow}

\newcommand{\be}{\begin{equation*}}
\newcommand{\ee}{\end{equation*}}
\newcommand{\bee}{\begin{equation}}
\newcommand{\eee}{\end{equation}}

\def\h{{\rm ht}}
\def\grade{{\rm grade}}

\def\Ass{{\rm Ass}}
\def\Min{{\rm Min}}

\def\pd{{\rm pd}}

\newcommand\res{\mathbb{F}} 

\newcommand{\Cmatroid}{{$C$-matroidal }}
\newcommand{\Cmatroidal}{{$C$-matroidal }}

\newcommand{\M}{\mathcal{M}} 

\title{Slightly mixed symbolic powers of matroids are locally glicci} 
\author[Paolo Mantero and Vinh Nguyen]{Paolo Mantero and Vinh Nguyen}

\begin{document}

\begin{abstract}
Let $\M$ be a matroid, and let $I_{\M}$ be either the Stanley--Reisner or the cover ideal of $\M$. In this paper we prove that for any matroid $\M$ on $[n]$, any $\ell\in \ZZ_+$, and any squarefree monomial $N\in R=\kk[x_1,\ldots,x_n]$, the ideal $I_{\M}^{(\ell)}:N$, which we call a ``slightly mixed symbolic power" of $I_{\M}$,  is always Cohen--Macaulay and locally glicci.
As a corollary, we obtain that all symbolic powers $I_{\M}^{(\ell)}$ are locally glicci. 
\end{abstract}

\maketitle
\section{Introduction}
A central open question in the theory of Gorenstein liaison is whether all Cohen--Macaulay ideals in a polynomial ring over a field are {\em glicci}, i.e. in the Gorenstein liaison class of a complete intersection \cite[Question~8.1]{KMMNP}. A large body of literature is dedicated to this problem, and various classes of Cohen--Macaulay ideals, in algebraic and geometric settings, have been shown to be glicci, see for instance \cite{CDSRVT}, \cite{CMR00}, \cite{CMR01}, \cite{DG09}, \cite{EHS}, \cite{Gaeta}, \cite{Gorla05}, \cite{Gorla07}, \cite{Gorla08}, \cite{Gorla10}, \cite{HHU}, \cite{HSS}, \cite{KR} and references therein. 

In addition to the above, strong evidence for a positive answer was provided. For instance, Casanellas, Drozd, and Hartshorne proved that all Gorenstein ideals are glicci \cite{CDH05}. Also, Migliore and Nagel proved that any Cohen--Macaulay generically Gorenstein ideal is glicci when viewed as an ideal in a polynomial ring with one more variable \cite{MN13}. 

On the other hand, the question is wide-open even for monomial ideals.

\begin{Conjecture}\label{Conj-glicci}
Let $\kk$ be a field.	If $I$ is a Cohen--Macaulay monomial ideal in $R=\kk[x_1,\ldots,x_n]$, then $I$ is glicci.
\end{Conjecture}
We single out a few positive results: 
\begin{enumerate}
    \item (Nagel--R\"omer \cite{NR08}) \Cref{Conj-glicci} holds if $I$ is the Stanley--Reisner ideal $I_\Delta$ of a weakly-vertex decomposable simplicial complex $\Delta$. 
       \item (Migliore--Nagel \cite{MN02}) \Cref{Conj-glicci}  holds if $I$ is any Cohen--Macaulay Borel-fixed ideal.
    \item (Huneke--Ulrich \cite{HU07}) If $|\kk|=\infty$, then the local version of \Cref{Conj-glicci} holds for any Artinian monomial ideal, i.e. $I_\m$ is glicci whenever $I$ is an $\m$-primary monomial ideal. Here $\m$ denotes the homogeneous maximal ideal of $R$. 
\end{enumerate}

Notably $(1)$ includes the case where $I$ is a \Cmatroidal ideal, meaning that $I$ is the cover ideal or the Stanley--Reisner ideal of a matroid, see \Cref{Def-Cmatroid}. Symbolic powers of \Cmatroidal ideals are also Cohen--Macaulay. In fact, Minh and Trung, and independently Varbaro, proved that \Cmatroidal ideals are precisely the squarefree monomial ideals whose symbolic powers are all Cohen--Macaulay \cite{TT} \cite{Var}. If $I$ is $C$-matroidal, it is then natural to investigate if $I^{(\ell)}$ is glicci too. In \cite{LM25}, it is proved that the polarization of $I^{(\ell)}$ is glicci for every $\ell$, and it is conjectured that the (symbolic) powers themselves are also glicci \cite[Cor.~3.3 and Conj.~3.4]{LM25}. A particular case of our main theorem proves their conjecture in the local setting. In fact, we first introduce the concept of ``slightly mixed symbolic powers" of a monomial ideal $J$. These are all the ideals of the form $\symp{J}:N$ for some squarefree monomial $N\in R$, see \Cref{slightly-mixed}. Taking $N=1$ one sees that this notion naturally generalizes the symbolic powers of $J$. Our first main result shows that these ideals are Cohen--Macaulay if $J$ is a \Cmatroidal ideal. 

\begin{Theorem}(\Cref{Mix-Sym-CM})
 Let $k$ be any field, $J \subseteq R = \kk[x_1,\ldots,x_n]$ any \Cmatroidal ideal, then $\symp{J}:N$ is Cohen--Macaulay, for any $\ell\in \ZZ_+$, and  any squarefree monomial $N$.  
\end{Theorem}

Our second main result closes the above question, \cref{Conj-glicci}, locally for symbolic powers of squarefree monomial ideals. Indeed, we prove the local version of \Cref{Conj-glicci} for all slightly mixed symbolic powers of \Cmatroidal ideals, and thus, the local version of  \cite[Conj.~3.4]{LM25}. 

\begin{Theorem}(\Cref{Locally-Glicci})
Assume $|\kk|=\infty$. Let $J \subseteq R = \kk[x_1,\ldots,x_n]$ be any \Cmatroidal ideal, and let $\m = (x_1,...,x_n)$. 
Then $\symp{J}:N$ is glicci in $R_\m$ for any $\ell\in \ZZ_+$, and  any squarefree monomial $N$. In particular, $J^{(\ell)}$ is glicci in $R_{\m}$, for any $\ell \geq 1$.  
\end{Theorem}

Our approach is to use  double $G$-linkage in an inductive manner by factoring out a variable $y$ at each step. This requires writing $J^{(\ell)} = I + yK$ for ideals $I,K$, where $I$ is Cohen--Macaulay and generically Gorenstein and $y$ is regular on $R/I$. However, there is a strong obstruction to this method. Normally, one takes $I$ to be the ideal generated by the monomials of $J^{(\ell)}$ not divisible by $y$. One can prove that such an ideal $I$ is the symbolic power of another \Cmatroidal ideal. Thus in particular $I$ is Cohen--Macaulay, but unfortunately, it is essentially never generically Gorenstein. In fact, this is why the general depolarization result of \cite{FKRS} cannot be applied to \cite[Cor.~3.3]{LM25} to prove \cite[Conj.~3.4]{LM25}. 

To overcome this difficulty, we construct a {\em Cohen--Macaulay and generically Gorenstein} ideal $I'$ so that $J^{(\ell)}= I+yK= I' +yK$, for some ideal $K$ and $y$ is regular on $R/I'$. Our method is
inspired by techniques used by Hartshorne \cite{Hartshorne66}, Geramita--Gregory--Roberts \cite{GGR86}, Migliore--Nagel \cite{MN00}, and Huneke--Ulrich \cite{HU07}. However, there are two main differences compared to these other constructions. First, our construction is very general, in the sense that it allows the lift of an ideal $I$ using general elements from any homogeneous ideal $K$. Secondly, the lifts in the above-mentioned results are all homogeneous ideals, while our lifts $I'$ of $I$ are no longer homogeneous. However, our lifts $I'$ still remain inside $\m$, explaining why we work in $R_\m$.   

The paper is structured as follows. \cref{Sec-2} discusses notation and preliminaries on matroids and Gorenstein liaison. In \cref{Sec-3}, we develop the lifting technique, and in \cref{Sec-4} we prove our main results.

\section{Preliminaries}\label{Sec-2}

In this short section we establish some notation and recall some  background on matroids and Gorenstein liaison. Throughout the paper we use the notation 
$$R = \kk[x_1,...,x_n] \qquad \text{ and } \qquad \m=(x_1,\ldots,x_n).$$

We briefly provide the definition of matroids and refer to \cite{Oxley} or \cite{Welsh} for further reading on matroids. Throughout we make the following shorthand, for a set $F$ and element $v$ we write $F - v$ for $F - \{v\}$ and $F \cup v$ for $F \cup \{v\}$.

\begin{Definition} A matroid $\M$ on a ground set $V$ consists of a non-empty collection $\B$ of subsets of $V$, such that $\B$ satisfies the following exchange axiom. For any $F,G \in \B$ and for any $v \in F - G$, there exist a $w \in G - F$ such that $(F - v) \cup w \in \B$. Elements of $\B$ are called the basis of $\M$.
\end{Definition}

By identifying a matroid with its independence complex, we may consider every matroid to be a simplicial complex (set the the basis of the matroid to be the facets of the simplicial complex). We now define two monomial ideals associated to simplicial complexes, and hence also to matroids.

\begin{Definition} Let $\Delta$ be a simplicial complex on $[n]$. Denote the set of facets of $\Delta$ as $\F(\Delta)$. For $F  \subseteq [n]$ we let $\p_F = ( x_i : i \in F) \subset R$. Then the Stanley--Reisner ideal $I_\Delta$ and the cover ideal $J(\Delta)$ associated to $\Delta$ are defined as 
$$I_\Delta = \bigcap_{[n] - F\in \F(\Delta)}\p_F \qquad \text{ and }\qquad J(\Delta) = \bigcap_{F\in \F(\Delta)}\p_F.$$
\end{Definition}

\begin{Definition}\label{Def-Cmatroid} 
An ideal $J\subseteq R$ is \Cmatroid if it is the Stanley--Reisner or the cover ideal of a matroid (see \cite[Def.~2.1]{MN1} for the significance of the ``$C$-" adornment).
\end{Definition}

\begin{Definition} 
Given a monomial $M\in R$ and a monomial ideal $J\subseteq R$, we let
\begin{itemize}
    \item $G(J)$ denote the set of minimal monomial generators of $J$;
    \item their {\em supports} to be the sets  
$$\supp{M} = \{x_i \,\mid\, x_i \text{ divides } M\}\quad \text{ and } \quad\supp{J}=\bigcup_{M\in G(J)}\supp{M}.$$
\end{itemize}
\end{Definition}

We now provide the definition of Gorenstein linkage. We refer to \cite{KMMNP} for further reading on Gorenstein liaison.

\begin{Definition} 
Two unmixed ideals $I,J \subset R$ of height $c$ are said to be {\em $G$-linked} if there is a Gorenstein ideal $H \subset I \cap J$ of height $c$ such that $J = H : I$.

The {\em Gorenstein liaison class} of $I$ consists of all ideals that are $G$-linked to $I$ in finitely many steps. The ideal $I$ is {\em glicci} if it is in the Gorenstein liaison class of a complete intersection.
\end{Definition}

We conclude with a local version of basic double $G$-linkage.

\begin{Theorem}\label{G-Double-Link}(\cite[Prop.~5.10]{KMMNP}, \cite[Prop.~4.1]{HU07}) Let $I$ and $K$ be ideals in $\m$, and $y \in \m$. If $R/I$ is Cohen--Macaulay and generically Gorenstein (i.e. $R_\p/I_\p$ is Gorenstein for all $\p\in \Ass(R/I)$), and $I + yK$ is unmixed of height greater than the height of $I$, then $(I+K)_{\m}$ and $(I+yK)_{\m}$ are Gorenstein linked in two steps in $R_{\m}$.    
\end{Theorem}

\section{Lifting Monomial Ideals with respect to other ideals}\label{Sec-3}

Let $\kk$ be an infinite field, $I_0\subseteq R_0:=\kk[x_1,\ldots,x_{n-1}]$ a monomial ideal, and let $y$ be a new variable and $R:=R_0[y]$.

In this section we discuss one of the main technical tools utilized in the paper -- a general method to lift $I$ to a radical ideal $I'$ in $R$. We have been inspired by other methods that have appeared in the literature in various papers. E.g. Hartshorne \cite{Hartshorne66}, Geramita--Gregory--Roberts \cite{GGR86}, and Migliore--Nagel \cite{MN00} all provide methods to ``lift" $I$ to a radical homogeneous ideal $I'$ in $R$. 
However, these methods have been developed with different mathematical objectives than ours, so it is not surprising that these lifts are not effective to solve our problem.

Instead, we introduce a more flexible method, allowing one to ``lift" all monomial ideals in $R_0$ via general elements of some homogeneous ideal $K\subseteq R$. We call it a {\em lift of $I_0$ relative to $K$} (see \Cref{Def-lift}). The main drawback of our method is that these lifts are not homogeneous ideals. 

\begin{Definition}
 For a monomial ideal $I$ in some polynomial ring over a field, let
$$
\max(I):=\max\{u\in \NN_0\,\mid\, \exists\,i \text{ s.t. } x_i^u \text{ divides some minimal generator of }I\}.
$$  
\end{Definition} So for instance, $I$ is squarefree $\Llra$ $\max(I)\leq 1$.

\begin{Construction}
Let $R_0=\kk[x_1,\ldots,x_{n-1}]$, $y$ a new variable, $R=R_0[y]$. Fix a proper homogeneous ideal $K\neq (0)$ in $R$ and $N\in \ZZ_+$. We define a map $\psi$, depending on $N$ and a minimal homogeneous generating set of $K$, which allows one to ``lift" all  monomial ideals of $R_0$ to ideals in $R$. 
Let $K$ be minimally generated by homogeneous elements $g_1,\ldots,g_u$. As we want to choose general elements in $K$ with some specific properties, we take any $N(n-1) \times u$ matrix $B$ with entries in $\kk$ satisfying property ($\star$) described below.

Let $\mathcal A$ be the power set of $\{x_1,...,x_{n-1}\}$. So, any $A\in \mathcal A$ is a (possibly empty) set with $\emptyset \subseteq A \subseteq \{x_1,...,x_{n-1}\}$. Set $u_A:=\mu(\ovl{K})\leq u$, where $\overline{K}$ is the image of $K$ in $\overline{R} = R/(A)$, and let 
$$
U=\{u_A\,\mid\,A\in \mathcal A\}.
$$
\begin{center}
($\star$)\qquad \qquad For every $u_A\in U$, every $u_A$-minor of $B$ is non-vanishing. Additionally, if $1\in U$, then all entries of $B$ are non-zero and distinct. 
\end{center}

\medskip

We observe that (i) property ($\star$) translates into finitely many non-vanishing conditions, (ii) a generic $N(n-1) \times u$ matrix $X$ satisfies ($\star$), and (iii) $|\kk|=\infty$. Therefore, a specialization of $X$ with general elements of $\kk$ preserves ($\star$), so such a matrix $B$ exists.

Now, letting $[\underline{g}]^T$ be the column vector whose entries are $g_1,\ldots,g_u$, we define elements $\underline{f}=f_{1,1},\,f_{1,2}\,,\,\ldots\,,\,f_{1,N}\,,\,f_{2,1}\,,\,\ldots,\,f_{n-1,1}\,,\,\ldots\,,\,f_{n-1,N}\in R$ as
$$B[\underline{g}]^T = [\underline{f}]^T.$$

We now use the $f_{i,j}$'s to construct the following $(n-1)\times N$ ``lifting matrix"
$$
\Psi=\begin{bmatrix} x_1 + yf_{1,1} & x_1 + yf_{1,2} &  \;\;\ldots\;\; & x_1 + yf_{1,N} \\ 
	x_2 + yf_{2,1}& x_2 + yf_{2,2} &\ldots & x_2 + yf_{2,N}   \\ 
	\vdots &\vdots &\vdots & \vdots \\ 
	x_{n-1} + yf_{n-1,1} & x_{n-1} + yf_{n-1,2} & \ldots & x_{n-1} + yf_{n-1,N}  
\end{bmatrix}.
$$
We let $\wdt{\Psi}$ be the (infinite) matrix, with $n-1$ rows, and columns indexed by $\ZZ_+$, obtained by extending $\Psi$ with zeroes for the $f_{i,j}$, i.e. setting $f_{i,j}=0$ for any $1\leq i \leq n-1$ and $j>N$. We employ this matrix to define a map $\psi$ from all the monomials of $R_0$ to $R$. First, we define $\psi$ on pure powers of monomials:
  $$
  \psi(x_{i}^a) = \prod_{j=1}^{a} (x_i+yf_{i,j}) \qquad \text{ for every }1\leq i \leq n-1.
  $$ We then extend this map multiplicatively to all monomials in $R_0$:
   $$
  \psi(x_1^{a_1}\cdots x_{n-1}^{a_{n-1}}) = \psi(x_1^{a_1})\cdots \psi(x_{n-1}^{a_{n-1}}).
  $$
 Note that $\psi$ cannot be extended to a ring homomorphism. Nevertheless, $\psi$ acts on every monomial of $R_0$. Hence in particular,  $\psi(G(I_0))\subseteq R$ is well-defined for any monomial ideal $I_0\subseteq R_0$.
\end{Construction}

We record a couple of properties of this lifting map.
\begin{Remark}\label{Rmk-min-gen}
In the above setting, for any $\emptyset \subseteq A \subseteq \{x_1,...,x_{n-1}\}$, we denote images in $\overline{R} = R/(A)$ by $\overline{\phantom{R}}$. For any such $A$, it follows from $(\star)$ that:
   \begin{enumerate}
	\item $\overline{K}$ is minimally generated by any $u_A$-subset of $\{\overline{f_{i,j}}\}$. In particular, when $A=\emptyset$, any $u$-subset of $\{f_{i,j}\}$ minimally generates $K$.
	
	\item If ${\overline{K}}\neq {\overline{(0)}}$, then $\overline{f_{i,j}} - \overline{f_{i,h}} \notin \overline{\m}{\overline{K}}$, for any $1\leq i\leq n-1$ and any $1\leq j< h\leq N$.  
\end{enumerate}
\end{Remark}

\begin{Definition}\label{Def-lift}
With the above setting and notation, let $I_0\subseteq R_0$ be a monomial ideal. We write 
$$
I'=(\psi(G(I_0))R.
$$
We call $I'$ a {\em partial lift of $I_0$ relative to $K$}. 
If $\max(I)\leq N$, then $I'$ is a {\em lift of $I_0$  relative to $K$}.
\end{Definition}

\begin{Remark}\label{Rmk-Basic-Facts-Lift}
Let $I',J'$ be partial lifts of proper monomial ideals $I_0,J_0\subsetneq R_0$, respectively, relative to a homogeneous ideal $K\subseteq R$. 
\begin{enumerate}[(i)]
    \item If $I_0\subseteq J_0$, then $I'\subseteq J'\subseteq \m=(x_1,\ldots,x_{n-1},y)$. 
    
    \item $I'$ is a homogeneous ideal $\Llra$ all entries of $\Psi$ are homogeneous $\Llra$ $K=R$.
    
    \item $(I_0,y)=(I',y)$.
\end{enumerate}
\end{Remark}

Recall that a quotient ring $R/J$ is $(S_h)$ (or satisfies {\em Serre's property }$(S_h)$) for some $h\in \ZZ_+$, if ${\rm depth}(R_\p/J_\p)\geq \max\{h, \dim(R_\p/ J_\p)\}$. E.g. $R/J$ satisfies $(S_1)$ if and only $\Ass(R/J)=\Min(J)$. 

\begin{Proposition}\label{Prop-x-regular}
Let $I_0 \subset R_0:=\kk[x_1,...,x_{n-1}] \subset R = R_0[y]$ be a monomial ideal, and let $I'$ be a lift of $I_0$ relative to some homogeneous ideal $K$ of $R$. Then	
	\begin{enumerate}
		\item $\h(I')=\h(I_0)$;
		\item for every $\q\in \Min(I')$ 
        there exist $d\in \ZZ_+$, $1\leq b_1<b_2<\ldots<b_d\leq n-1$ and $1\leq i_j \leq N$ such that
		$$	\q=(x_{b_1} + yf_{b_1,i_1}\,,\; x_{b_2} + yf_{b_2,i_2}\,,\ldots,\;x_{b_d} + yf_{b_d,i_d} );$$
 		\item $y\notin \bigcup_{\q\in \Min(I')}\q$. In particular, if $R/I'$ is $(S_1)$ (e.g. if $R/I_0$ is $(S_1)$, see \Cref{Lift-Algebraic-Properties}(3)), then $y$ is regular on $R/I'$.


	\end{enumerate}
\end{Proposition}

\begin{proof}
Let $c:=\h(I_0)$.\\
(1) Let $\p_0\in \Min(I_0)$ with $\h(\p_0)=\h(I_0)$. After possibly reordering the variables we may assume $\p_0=(x_1,\ldots,x_c)$. By \Cref{Rmk-Basic-Facts-Lift}$(i)$, $I'\subseteq \p_0'=(x_1+yf_{1,1},x_2+yf_{2,1},\ldots,x_c+yf_{c,1})$, and $\p_0'$ is obviously a prime ideal of height $d$ (e.g. because $R/\p_0'\cong \kk[x_{c+1},\ldots,x_{n-1},y]$). This proves $\h(I')\leq c$. 
	Now, let $\q$ be a minimal prime of $I'$ with $\h(\q)=\h(I')\leq c$. Then $(\q,y)\supseteq (I',y)=(I_0,y)R$. Since $I_0\subseteq R_0$, then 
    $$
    c+1=\h((I_0,y))\leq \h((\q,y))\leq \h(\q)+1=\h(I')+1\leq c+1,
    $$ 
    where the rightmost inequality holds by Krull's Altitude Theorem and because $R$ is a catenary ring. This implies that $\h(I')=c$.	
	
	(2) Since $\psi(G(I_0))$ is a minimal generating set of $I'$, 
    and each element in  $\psi(G(I_0))$ is a product of polynomials of the form $x_b - yf_{b,j}$, then every $\q\in \Min(I')$ is also a minimal prime of some ideal $H\supseteq I'$ generated by some of these polynomials $x_b - yf_{b,j}$.
    
    We will trim $H$ to a prime ideal containing $I'$. Let $B$ be the set of indices $b$ where some $x_b-yf_{b,j} \in H$. For any given $b \in B$, we set 
    $$
    k_b =\min\{ j \,\mid\, x_b - yf_{b,j} \in H\}.
    $$

    Now we let $H' =(x_{b_1} + yf_{b_1,k_1},x_{b_2} + yf_{b_2,k_2},\ldots,x_{k_d} + yf_{b_d,k_d} )$, where $1\leq b_1<b_2<\ldots<b_d \leq n-1$. As above, $H'$ is prime. Also, $I' \subseteq H'$ because, by construction of the map $\psi$, any element in $\psi(G(I_0))$ divisible by $x_b - yf_{b,j}$ is also divisible by $x_b - yf_{b,k_b}$ since $k_b \leq j$. Therefore, $I'\subseteq H' \subseteq H \subseteq \q$. Since $\h(\q)=\h(H)=\h(I')$, then $\h(\q)=\h(H')$ too. Since $H'\subseteq \q$ are two prime ideals of the same height $H'=\q$. 

(3) Let $\q \in \Min(I')$, then, by part (2), $\q=(x_{b_1} + yf_{b_1,i_1},x_{b_2} + yf_{b_2,i_2},\ldots,x_{b_d} + yf_{b_d,i_d} )$ and, as explained in part (1), $\h(\q)=d$. If $y\in \q$, then $\q=(x_{b_1},x_{b_2},\ldots,x_{b_d},y)$ would have height $d+1$, yielding a contradiction.
\end{proof}

We recall the notion of ``expected ranks" from a paper of Buchsbaum and Eisenbud \cite{BE73}.
\begin{Definition}
   Let $(\res_\bullet, d_\bullet)$ be a finite complex of free $R$-modules, and let $F_0,\ldots,F_p$ be the free modules in it. For any $0\leq i \leq p$, one defines 
   $$
   r_i(\res_\bullet):=\sum_{j=0}^{p-i}(-1)^{j}\,\textrm{rank}\,F_{i+j}.$$\end{Definition}
It is well--known that the ideals of $r_i$-minors, $I_{r_i}(d_i)$, of the matrices of $\res_\bullet$ detect whether $\res_\bullet$ is a free resolution. By the following result, these minors detect Serre's properties too. 
The result is folklore, and can be found in \cite[Prop.~6.2.3]{Kummini}, see also \cite[6.11.3]{EGAIVPart2}.

However, for the sake of completeness, we sketch a short proof below. The proof follows along the same lines as \cite[Prop.~2.4]{HMMS}. Recall that for a finite $R$-module $M$ $\h(M):=\h({\rm ann}(M))$.

\begin{Proposition}\label{ShProperty}
Let $M$ be a finite $R$-module. Assume $M$ is (grade-)unmixed, i.e. $\h(\p)=\h(M)$ for all $\p \in V({\rm ann}(M))$. 	Let $(\res_\bullet, d_\bullet)$ be a free resolution of $M$. 

For any $h\in \ZZ_+$, $M$ satisfies Serre's condition $(S_h)$ if and only if
	\[\h(I_{r_i}(d_i)) \ge \min\{\dim(R), i+h\} \qquad \text{for all $i = \h(M)+1,\ldots, p$.}\]
\end{Proposition}

\begin{proof} The statement follows from these equivalences:
	$$\begin{array}{ll}
		M \text{ is not } (S_h) 	&  \Leftrightarrow  \exists\, \p \in \supp{M} \text{ with } {\rm depth}(M_\p) < \min\{h, \dim(M_\p)\}\\
		& \Leftrightarrow \exists\, \p \in \supp{M}\text{ with } \pd_{R_\p}(M_\p) > \max\{\h(\p) - h, \h(\p) - \dim(M_\p)\}\\
		& \Leftrightarrow \exists\, \p \in \supp{M}\text{ with } \pd_{R_\p}(M_\p) > \max\{\h(\p) - h, \h(M)\}\\
		&  \Leftrightarrow \exists\, \p \in \supp{M}\text{ with }\p \supseteq I_{r_{t+1}}(d_{t+1}), \text{ where } t := \max\{\h(\p) - h, \h(M)\}\\
		& \Leftrightarrow \h(I_{r_{t+1}}(d_{t+1})) \leq \min\{\dim(R),t + h\} \text{ for some } t \geq \h(M). 
	\end{array}$$
The second equivalence follows from the Auslander-Buchsbaum Theorem.  The third equivalence follows from  $\dim(M_\p)=\h(\p)-\h(M)$ and the unmixedness of $M$.  The fourth equivalence follows because the map $(d_{t+1})_\p$ splits if and only if $\p \not\supset I_{r_{t+1}}(d_{t+1})$ if and only if  $\pd_{R_\p}(M_\p) \le t$. 

Now for last equivalence. For the forward direction we go through cases depending on the value of $t$. If $t = \h(M)$, then $\h(\p) \leq \h(M) + h$, it follows that $\h(I_{r_{t+1}}(d_{t+1})) \leq \h(\p) = t +h$. For the other case, $t = \h(\p) -h \geq \h(M)$. Then $\h (I_{r_{t+1}}(d_{t+1})) \leq \h(\p)\leq t + h$, and indeed $t \geq \h(M)$. For the backwards direction, take any prime $\p$ with $\p\supseteq I_{r_{t+1}}(d_{t+1})$ and $\h(\p)=\h( I_{r_{t+1}}(d_{t+1}))$. By assumption $\h(\p) -h\leq t$ and $t \geq \h(M)$, thus $t \geq \max\{ \h(\p) -h, \h(M) \} =: t'$. The result follows by the inclusion $\sqrt{I_{r_{t'+1}}(d_{t'+1})} \subseteq \sqrt{I_{r_{t+1}}(d_{t+1})}$ as $t' \leq t$, see \cite[p.~261]{BE73}.
\end{proof}

Recall that an ideal $J$ in $R$ is called {\em unmixed} if $\h(\p)=\h(J)$ for every $\p\in \Ass(R/J)$. In particular, if $J$ is unmixed, then $R/J$ is $(S_1)$.

A key result for us will be showing that our construction produces a radical ideal $I'$ whenever $I_0$ is unmixed and $\h\, K > \h \, I_0$. We also show that a free resolution of $I'$, which will be a minimal resolution locally at the maximal homogeneous ideal, is obtained by properly lifting a minimal free resolution of $I_0$. To the best of our knowledge, this procedure first appeared in \cite{MN00}. Since our construction is similar but incomparable with the one of {\it op. cit.}, below we give a short adaptation of their proof that fits our needs. For more details on this procedure, we refer to the proof of \cite[Prop.~2.6]{MN00} and the discussion surrounding it.

\begin{Theorem}\label{Lift-Algebraic-Properties}
Let $I_0 \subset R_0:=\kk[x_1,...,x_{n-1}] \subset R = R_0[y]$ be a monomial ideal, and let $I'$ be a lift of $I_0$ relative to some homogeneous ideal $K$ of $R$ and some $N\geq \max(I_0)$. Then
\begin{enumerate}
    \item $\beta_i^{R_\m}(R_\m/I'_\m) = \beta_i^{R_0}(R_0/I_0)$ for all $i$, where $\m:=(x_1,\ldots,x_{n-1},y)\subseteq R$.\\
    In particular, $I'_\m$ is Cohen--Macaulay, Gorenstein, or a complete intersection $\Llra$ $I_0$ is respectively.

    \item $y$ is regular on $R_{\m}/I'_{\m}$.
    
    \item If $R_0/I_0$ is $(S_h)$ for some $h\geq 1$ $\Lra$ $R/I'$ is $(S_h)$.
    
   \item If $I_0$ is unmixed $\Lra$ $I'$ is unmixed.
    
    \item\label{lift-is-radical} If $I_0$ is unmixed and $\h(I_0)<\h(K)$ $\Lra$ $I'$ is a radical ideal. 

\end{enumerate}
\end{Theorem}

\begin{proof}
For a matrix $\varphi$, denote by $I_{\ell}(\varphi)$ to be the ideal generated by the $\ell \times \ell$ minors of $\varphi$.\\

\noindent (1) We start by making the following claim. Fix a minimal multi-graded free resolution of $R_0/I_0$ over $R$, $$\res_{\bullet}: 0 \xrightarrow{} F_{p} \xrightarrow{\varphi_p} F_{p-1} \xrightarrow{\varphi_{p-1}} \cdots \xrightarrow{\varphi_2} F_1 \xrightarrow{\varphi_1} F_0 \xrightarrow{} R/I \to 0.$$ Then $\res_{\bullet}$ can be lifted to a free resolution $$\res_{\bullet}': 0 \xrightarrow{} F_{p} \xrightarrow{\varphi_p'} F_{p-1} \xrightarrow{\varphi_{p-1}'} \cdots \xrightarrow{\varphi_2'} F_1 \xrightarrow{\varphi_1'} F_0 \xrightarrow{} R/I' \to 0,$$ where every entry of every $\varphi_i'$ is in $\m$, the image in $R/(x)\cong R_0$ of each map $\varphi_i'$ is precisely $\varphi_i$, and $\h(I_\ell(\varphi_i'))\geq \h(I_\ell(\varphi_i))$ for all $i$ and $\ell$.

To establish the claim, we observe that since $\res_\bullet$ is multi-graded, then every entry of every map $\varphi_i$ could be ``lifted" via the map $\psi$. However, unfortunately, doing so may not result anymore in a complex. Regardless, there is still a way to apply variations of the lifting map $\psi$ to obtain a {\em complex}. One first lift the entries of $\varphi_1$ using $\psi$. Next, assuming the entries of $\varphi_{i-1}$ have been lifted using some map, one lifts $\varphi_i$ by mapping its entries $\prod_{j}x_j^{a_j}$ to $\prod_{j,k} x_j+yf_{j,k}$ in a similar way to $\psi$, except that the $k$'s may be larger than the ones of $\psi$. 
This is explained, for instance, in \cite[Proof of Prop.~2.6]{MN00}. As all these variations of the lifting map still send $x_j^a$ to a product of $a$-many polynomials of the form $x_j + yf_{j,k}$'s, then each entry of the matrix $\varphi_i'$ is still equal to the corresponding entry of $\varphi_i$ modulo $y$.

We now prove that for any $\ell$ and $i$, one has $\h(I_\ell(\varphi_i'))\geq \h(I_\ell(\varphi_i))$. Fix such $\ell$ and $i$. Let $\q$ be a prime of $I_\ell(\varphi_i')$ with $\h(\q)=\h(I_\ell(\varphi_i'))$. 
Since each entry of $\varphi_i'$ equals the corresponding entry of $\varphi_i$ modulo $y$, then 
$(I_\ell(\varphi_i),y) = (I_\ell(\varphi_i'),y)$. Hence
	$$\h(I_\ell(\varphi_i))+1=\h((I_\ell(\varphi_i),y)) = \h((I_\ell(\varphi_i'),y)) \leq \h(I_\ell(\varphi_i'))+1=\h(\q)+1.
	$$
	The leftmost equality holds because $I_\ell(\varphi_i)R$ is extended from $R_0$, the next equality holds because $(I_\ell(\varphi_i'),y)) = (I_\ell(\varphi_i),y)$. The inequality holds by Krull's Altitude Theorem and because $R$ is catenary. 
	Thus, $\h(I_\ell(\varphi_i'))=\h(\q)\geq \h(I_\ell(\varphi_i))$. 

 Next, by the acyclicity criterion of Buchsbaum and Eisenbud \cite{BE73}, our complex $\res_{\bullet}$ is acyclic if and only if $\grade\, I(\varphi_i') \geq i$ for all $i$, where $I(\varphi_i') := I_{r_i'}(\varphi_i')$ and $r_i':=r_i(\res_\bullet')$. Now clearly $r_i(\res_{\bullet}) = r_i(\res_{\bullet}')$, and by the above, $\grade \, I(\varphi_i')=\h\, I(\varphi_i')\geq \h\, I(\varphi_i)=\grade \, I(\varphi_i)  \geq i$ for $1 \leq i \leq p$, where the equalities hold because $R$ is Cohen--Macaulay. 
Therefore, $\res_{\bullet}'$ is a resolution of $R/I'$.
	
 By construction of $\varphi_i'$, all entries of $\varphi_i'$ are in $(x_1,...,x_{n-1},y)$, so $(\res_{\bullet}')_\m$ is a minimal free resolution of $R_\m/I'_\m$. The statement about Betti numbers and the in particulars follow.\\ 

\noindent (2) By \cite[Cor.~1.6.19]{Bruns-Herzog} $y$ is regular on $R_{\m}/I'_{\m}$ if and only if $H_i(y;R_{\m}/I'_{\m}) =0$ for $i > 0$, where $H_i(y;R_{\m}/I'_{\m})$ denote Koszul homology of $R_{\m}/I'_{\m}$ with respect to the sequence $y$. Since $y$ is regular on $R_\m$ there is an isomorphism $H_i(y;R_{\m}/I'_{\m}) \cong \Tor_i^{R_\m}(R_\m/(y), R_{\m}/I'_{\m})$. By $(1)$ $(\res'_{\bullet})_\m$ is a minimal free resolution of $R_{\m}/I'_{\m}$. Now since  $R_\m/(y)\otimes(\res'_\bullet)_\m = (\res_\bullet)_\m$ and $(\res_\bullet)_\m$ is acyclic, we have $\Tor_i^{R_\m}(R_\m/(y); R_{\m}/I'_{\m}) =0$ for $i > 0$.\\

\noindent (3) This statement follows by (1) and \Cref{ShProperty}, because $\h I(\varphi_i')\geq \h I(\varphi_i)$. \bigskip

For the next two parts, write $c:=\h(I_0)=\h(I')$. \\
\\
(4) Since $R_0/I_0$ is $(S_1)$ then, by (3), $R/I'$ is $(S_1)$ too, i.e. $\Ass(R/I')=\Min(I')$. Let $\q\in \Min(I')$. By \Cref{Prop-x-regular}(2), $(\q,y)=(x_{b_1},\ldots,x_{b_d},y)$ for some $d\geq c$. Let $\q_0=(x_{b_1},\ldots,x_{b_d})$. Since $I_0\subseteq (I_0,y)=(I',y)\subseteq (\q,y)=(\q_0,y)$ and $I_0,\q_0$ are $R_0$-ideals, then $I_0\subseteq \q_0$. To prove this statement, it suffices to show that $\q_0\in \Min(I_0)$. Since then 
$$
\h(\q)=\h((\q,y))-1=\h(\q_0,y)-1=\h(\q_0)=c,
$$ where the leftmost equality follows from \Cref{Prop-x-regular}(3), and the rightmost equality follows from the assumption that $I_0$ is unmixed. 

To prove $\q_0\in \Min(I_0)$, we first observe that $(\q,y)=(x_{b_1},\ldots,x_{b_d},y)\in {\rm Spec}(R)$. Also, by \Cref{Prop-x-regular}(3), $\h(\q,y)=\h(\q)+1$, thus 
$$
\begin{array}{lll}
\dim((R/I')_{(\q,y)}) & =\h((\q,y)_{(\q,y)}) - \h(I'_{(\q,y)}) & = \h((\q,y)) - \h(I'_{(\q,y)})\\
& = 1+\h(\q) - \h(I'_{(\q,y)}) & \leq 1+(\h(\q) - \h(I'_\q))\\
& =1+\dim((R/I')_\q)& =1+0\\
& =1,& 
\end{array}
$$
where the second to last equality holds because $\q\in \Min(I')$. 
Since $y$ is in $(\q,y)$ and, by \Cref{Prop-x-regular}(3), $y$ is a   regular element of the local ring $(R/I')_{(\q,y)}$, then 
$$
\dim((R/I')_{(\q,y)})=1.
$$
Therefore, $\dim((R/(I',y))_{(\q,y)})=0$, and since 
$$
(R/(I',y))_{(\q,y)}=(R/(I_0,y))_{(\q_0,y)}\cong (R_0/I_0)_{\q_0},
$$
so $(R_0/I_0)_{\q_0}$ is Artinian, i.e. $\q_0\in \Min(I_0)$, concluding the proof of this part. \\
\\
(5) We prove that $R/I'$ is $(S_1)$ and $I'R_{\q} = \q R_{\q}$ for any $\q \in \Ass(R/I')$. First, since $I_0$ is unmixed, then by (4) and \Cref{Prop-x-regular}(1), $I'$ is unmixed of height $c$, and in particular $R/I'$ is $(S_1)$. Now, fix any $\q\in \Ass(R/I') = \Min(I')$. 
By \Cref{Prop-x-regular}(2), $\q=(x_{b_1} + yf_{b_1,i_1},x_{b_2} + yf_{b_2,i_2},\ldots,x_{b_c} + yf_{b_c,i_c} )$ for some $1\leq b_1<b_2<\ldots<b_c\leq n-1$ and $1\leq i_j \leq N$, so $\q$ contains at least one entry from $c$ distinct rows of the lifting matrix $\Psi$.

To show that $I'R_{\q} = \q R_{\q}$ we need to show that $\q$ contains exactly one entry from each of these rows. Suppose not, then $x_{b_j}+yf_{b_j,i_j}$ and $x_{b_j}+yf_{b_j,k_j} \in \q$ for some $1\leq i_j<k_j\leq N$, so $y(f_{b_j,i_j}-f_{b_j,k_j})\in \q$. By \Cref{Prop-x-regular}, $y\notin \q$, so $f_{b_j,i_j}-f_{b_j,k_j} \in \q$. Let $\overline{R} = R/(x_{b_1},...,x_{b_c})$, and let $\overline{\phantom{R}}$ denote images in $\overline{R}$. Then in $\overline{R}$ we have the following containment,
$$\overline{f_{b_j,i_j}} - \overline{f_{b_j,k_j}} \in\overline{\q} = (\overline{yf_{b_1,i_1}}, \overline{yf_{b_2,i_2}}, ... , \overline{yf_{b_c,i_c}}) \subseteq \overline{y}\overline{K}.$$

By \Cref{Rmk-min-gen}(2), this implies $\overline{K} = \overline{0}$, i.e. $K \subseteq (x_{b_1},\ldots,x_{b_c})$, contradicting that $\h \, K > c$. \end{proof}

\section{Slightly Mixed Symbolic Powers}\label{Sec-4}

In this section we prove our main result. First, we introduce the notion of slightly mixed symbolic powers. 
\begin{Definition}\label{slightly-mixed}A {\em slightly mixed symbolic power} of a monomial ideal $\wdt{J}$ is any ideal of the form 
$$J = \symp{\wdt{J}} : N$$
for some $\ell\in \ZZ_+$ and some squarefree monomial $N\in R$. 
\end{Definition}
Clearly, every symbolic power $\symp{\wdt{J}}$ is a slightly mixed symbolic power -- one just takes $N=1$.

Let us recall that Minh--Trung \cite{MT} and Varbaro \cite{Var} proved, independently, that symbolic powers of \Cmatroidal ideals are Cohen--Macaulay. We now extends this result by proving that all  {\em slightly mixed symbolic powers} of \Cmatroidal ideals are Cohen--Macaulay. 

\begin{Theorem}\label{Mix-Sym-CM} 
Let $J \subset R$ be a \Cmatroid ideal, then every slightly mixed symbolic power of $J$ is Cohen--Macaulay.
\end{Theorem}

\begin{proof} 
We first show that taking the restriction of slightly mixed symbolic powers is the same as taking slightly mixed symbolic powers of the restriction of $J$. More generally, for any monomial ideal $I$ of the form $I=\bigcap_{\p\in \Ass(R/I)}\p^{a_\p}$, where $a_\p \in \ZZ_+$ for all $\p\in \Ass(R/I)$, any squarefree monomial $N$, and any variable $y \notin \supp{N}$, writing $R=R_0[y]$, we prove that 
$$
(\star)\qquad (I:_R N)\cap R_0 = (I\cap R_0) :_{R_0} N.
$$
To see it, let us recall that the {\em order} $O_{\p}(f)$ of an element $f\in R$ at a prime ideal $\p$ is the order of $f$ in the $\p$-adic evaluation, i.e.  $O_{\p}(f) = \max\{h\in \NN_0 : f \in \p^h\}.$ Then
$$
\begin{array}{ll}
(I:_RN)\cap R_0 & = (\bigcap_{\p \in \Ass(R/I)}\p^{\max\{a_\p - O_\p(N) ,0\}}) \cap R_0 = \bigcap_{\p \in \Ass(R/I)}(\p\cap R_0)^{\max\{a_\p - O_\p(N) ,0\}}\\
& = (\bigcap_{\p \in \Ass(R/I)}(\p\cap R_0)^{a_\p}):_{R_0}N = (I\cap R_0) :_{R_0}N,
\end{array}
$$
where the first equality in the second row holds because  $y\notin \supp{N}$, thus $O_\p(N) = O_{\p\cap R_0}(N)$ for every $\p\in \Ass(R/I)$.
\medskip

    $(2)$ Write $I := \symp{J}$. We show $I: N$ is Cohen--Macaulay by induction on $\deg N\geq 0$. If $\deg(N)=0$, then $N=1$, so $I:_R N=I$, which is Cohen--Macaulay by \cite{MT} and \cite{Var}. Next, assume $\deg N = 1$, i.e. $N = y$ for some variable $y$. By the above, we may assume $y\in \supp{J}$, so $y$ is not regular on $R/I$. Since $R/I$ is Cohen--Macaulay, then $\h((I,y))=\h(I)$.   Consider the exact sequence 
    $$
    0 \to R/(I :_R y) \to R/I \to R/(I,y) \to 0.
    $$
    As $R/I$ is Cohen--Macaulay and $\h(I,y)=\h(I)$, we see that $R/(I:_Ry)$ is Cohen--Macaulay if and only if $R/(I,y)$ is Cohen--Macaulay. Now, we observe that $(I,y) = (I\cap R_0)R + (y)=(J\cap R_0)^{(\ell)}R + (y)$. 
 The ideal $J\cap R_0$ is $C$-matroidal, because if we write $y=x_i$ for some $i$, and $J$ as the Stanley--Reisner ideal of a matroid $\M$, then $J\cap R_0$ is the Stanley--Reisner ideal  of the restriction of $\M$ to $[n]-\{i\}$. Therefore, $I\cap R_0$ is Cohen--Macaulay, again by \cite{MT} and \cite{Var}. Since $y$ is regular on $(I\cap R_0)R$, the result follows.

    If $\deg(N)>1$, we write $N':=N/y$ for some $y\in \supp{N}$. Since $N'$ is squarefree, then $y\notin \supp{N'}$. Note that $ I :_R N = (I :_R N') :_R y$. Consider the exact sequence $$0 \to R/((I :_R N') :_R y) \to R/(I :_R N') \to R/(I :_R N',y) \to 0.
    $$

        By induction $R/(I :_R N')$ is Cohen--Macaulay, hence it suffices to show that $R/(I : N',y)$ is Cohen--Macaulay. We write 
        $(I :_R N',y) = ((I :_R N')\cap R_0)R + (y)$, which by ($\star$), equals $((I\cap R_0) :_{R_0} N')R + (y)$. Since  $(I\cap R_0) :_{R_0} N'$ is a slightly mixed symbolic power of a \Cmatroid ideal, then it is Cohen--Macaulay by induction on deg $N'$. Again, we are done because $y$ is a non-zero divisor on $((I : N')\cap R_0)R$ .
\end{proof}

We can now prove the main theorem of this paper.
\begin{Theorem}\label{Locally-Glicci}
Let $\kk$ be an infinite field, $R = \kk[x_1,...,x_n]$ and $\m = (x_1,...,x_{n})$. For any \Cmatroidal ideal $\wdt{J}$, every slightly mixed symbolic power of $\wdt{J}$ is glicci in $R_\m$.  In particular all symbolic powers of \Cmatroid ideals are glicci in $R_\m$.
\end{Theorem}

\begin{proof} Let $J = \symp{\wdt{J}} : N$, and let $N' = \prod_{x_i\in \supp{N}\cap \supp{J}} x_i$. Then $\symp{\wdt{J}} : N = \symp{\wdt{J}} : N'$, so throughout we can always assume $\supp{N} \subseteq \supp{J}$.  

Now, $J$ has the following primary decomposition 
$$
J = \bigcap_{\p \in \Ass(R/\wdt{J})} \p^{\max\{0, \ell- O_\p(N)\}}.
$$
Let $a_\p:=\max\{0, \ell- O_\p(N)\}\in \NN_0$. We induct on $\sum_{\p \in \Ass(R/\wdt{J})} a_\p\geq 1$. If $\sum_{\p \in \Ass(R/\wdt{J})} a_\p = 1$, then $J$ is a complete intersection so there is nothing to show.

Assume that $\sum_{\p \in \Ass(R/\wdt{J})} a_\p >  1$. First consider the case where $N = x_1\cdots x_n$. If $\ell\leq \h(\wdt{J})$, then $J=R$, and if $\ell> \h(\wdt{J})$ then $J = \wdt{J}^{(\ell - \h\, J)}$, which is glicci in $R_\m$ by induction. 

We may then assume that there exists a variable $y \notin \supp N$. We write 
$$
J = I + yK,  \qquad \text{ where } \; K = J : y \; \textrm{ and } \; I = (J\cap R_0)R.
$$
Note that $K \supseteq J$ hence $\h \, K \geq \h \, J > \h \, I = c - 1$. Let $I'$ be a lift of $I$ relative to $K$. From the construction, it follows that $I + yK = I' + yK$. Also, $I' \subseteq J \subseteq K$. By \cref{Mix-Sym-CM} $I$ is Cohen--Macaulay and then, by \Cref{Lift-Algebraic-Properties}(1), $I'$ is Cohen--Macaulay too. As $\h\, K > \h \, I$ we may then apply \Cref{Lift-Algebraic-Properties}(4) to obtain that $I'$ is radical. In particular, $I'$ is generically Gorenstein.  As $I' \subseteq K$, applying \cref{G-Double-Link} we find that $J_\m = (I + yK)_\m = (I' + yK)_\m$ is G-linked in two steps to $(I' + K)_\m = K_\m$ in $R_\m$. Finally we argue that $K_\m$ is glicci. Recall $K = J : y$, hence 
$$
K = \bigcap_{\p \in \Ass(R/\wdt{J})} \p^{\max\{0, a_\p'\}},
$$
where $a_\p'=a_\p - O_\p(y)$. Since $y\in \supp{J}$, then $a_\p'<a_\p$ for some $\p \in \Ass(R/\wdt{J})$. Thus, by induction, $K_\m$ is glicci, and therefore $J_\m$ is glicci too.\end{proof}

\bibliographystyle{abbrv}
\bibliography{GlicciMatroids}

\end{document}